\documentclass[12pt]{amsart}

 \usepackage[doublespacing]{setspace}
\usepackage[utf8]{inputenc}
\usepackage{graphicx}
\usepackage{amscd}
\usepackage{amsmath}
\usepackage{amssymb}
\usepackage{amsthm}
\usepackage{amsfonts}
\usepackage{fancyhdr}
\usepackage{mathrsfs}
\usepackage{amsxtra}
\usepackage{MnSymbol} 

\usepackage{setspace}
\usepackage{mathtools}

\newtheorem{theorem}{Theorem}
\newtheorem{lemma}{Lemma}
\newtheorem{corollary}{Corollary}
\newtheorem{proposition}{Proposition}

\newtheorem{example}{Example}
\newtheorem{remark}{Remark}

\def\beq{\begin{eqnarray*}}
\def\eeq{\end{eqnarray*}}
\def\bt{\begin{theorem}}
\def\et{\end{theorem}}
\def\bp{\begin{proposition}}
\def\ep{\end{proposition}}
\def\bl{\begin{lemma}}
\def\el{\end{lemma}}

\def\bp{{\bar\partial}}

\def\MM{{\mathcal M}}

\def\ker{\text{\rm ker}}
\def\dim{\text{\rm dim}}

\usepackage{color}

\def\R{{\mathbb R}}

\title{A Note on multipliers between model spaces}
\author[Fricain]{Emmanuel Fricain}
\address{Laboratoire Paul Painlev\'e, Universit\'e Lille 1, 59 655 Villeneuve d'Ascq C\'edex }
\email{emmanuel.fricain@math.univ-lille1.fr}

\author[Rupam]{Rishika Rupam}
\address{Laboratoire Paul Painlev\'e, Universit\'e Lille 1, 59 655 Villeneuve d'Ascq C\'edex}
\email{rishika.rupam@math.univ-lille1.fr}

\thanks{The authors were supported by Labex CEMPI (ANR-11-LABX-0007-01)}


\date{}

\begin{document}
\maketitle

 \begin{abstract} 
 In this note, we study the multipliers from one model space to another. In the case when the corresponding inner functions are meromorphic, we give both necessary and sufficient conditions ensuring this set of multipliers is not trivial. Our conditions involve the Beurling--Malliavin densities and are based on the deep work of Makarov--Poltoratski on injectivity of Toeplitz operators.
\end{abstract}

\section{Introduction}
 For a pair of inner functions $U$ and $V$ on the upper half-plane $\mathbb C_+=\{z\in\mathbb C:\Im{\rm{m}}(z)>0\}$, the multipliers set $\mathcal M(U,V)$ is the set of analytic functions $\Phi$ on $\mathbb C_+$ such that $$
\Phi K_U\subset K_V.
$$
Here $K_U$ (respectively $K_V$) is the model space associated to $U$ (respectively to $V$). See Section~\ref{sec:def-model} for the definition. A basic question here is whether or not
$$
\mathcal M(U,V)\neq \{0\}?
$$

A source of inspiration for this paper stems from \cite{GMR,MR3616198} which examined various pre-orders on the set
of partial isometries and contractions on Hilbert spaces and their relationship to their associated Liv\v{s}ic characteristic functions. It
turns out, for example, that when the Liv\v{s}ic characteristic functions $u$ and $v$ for two partial isometries $A$ and $B$ are inner (on the unit disc), the issue of whether or not $A$ is "less than" $B$ can be rephrased as to whether or $\mathcal M(u,v)\neq\{0\}$. Another motivation comes from the work of Crofoot \cite{Crofoot} who studied the onto multipliers. 

In \cite{fricain2016multipliers}, the authors characterize the multipliers from one model space to another in terms of kernels of Toeplitz operators and Carleson measures for model spaces.  However, it is widely understood that both the injectivity problem of Toeplitz operators and the Carleson measures question  for model spaces are rather difficult. As a result, it is not easy to apply the characterization obtained in \cite{fricain2016multipliers} in concrete situations. In this paper, we pursue this line of research. We consider the case when $U$ and $V$ are both meromorphic on $\mathbb C$. Our aim is to simplify the characterization proved in \cite{fricain2016multipliers} and to apply it to several examples. 

\section{Preliminaries} 

\subsection{Basic notations}
We use the standard notation $\mathcal H^p=\mathcal H^p(\mathbb C_+)$, $1\leq p\leq \infty$, for the Hardy space of the upper half-plane and as usual we identify functions in $\mathcal H^p$ with their boundary values on $\mathbb R$. We denote by $\Pi$ the Poisson measure on $\mathbb R$, 
$$
d\Pi(t)=\frac{dt}{1+t^2},
$$
and by $L^1_\Pi=L^1(\mathbb R,\Pi)$. The Hilbert transform of a function $h\in L^1_\Pi$ is defined as the singular integral
$$
\tilde h(x)= \lim_{\varepsilon\to 0} \frac{1}{\pi} \int_{|x-t|>\varepsilon}\left[ \frac{1}{x-t}+\frac{t}{1+t^2} \right] h(t)\,dt.
$$
Recall that outer functions $H$ are of the form 
$$
H=e^{h+i\tilde h}\qquad \mbox{on }\mathbb R,
$$ 
for some $h\in L^1_\Pi$. Recall also that if $h\in L^1_\Pi$, then $\tilde h\in L_\Pi^{o(1,\infty)}$ (the weak $L^1$ space), {\it i.e.}
$$
\Pi\{|\tilde h|>A\}=o\left(\frac{1}{A}\right),\qquad A\to \infty.
$$
See \cite[Corollary 14.6]{MR2500010}. 

We shall need the elementary Blaschke factor on $\mathbb C_{+}$ with zero at $i$: 
$$b_{i}(z) := \frac{z - i}{z + i},$$
and 
$$k_i(z)=\frac{1}{\pi}\frac{1}{z+i},$$
the corresponding  kernel (of $\mathcal H^2$) at $i$. 
\subsection{Meromorphic Inner Functions and model spaces}\label{sec:def-model}
Recall that an inner function $U$ on the upper half-plane is a bounded and analytic function on $\mathbb C_+$ with boundary values of modulus one almost everywhere on $\mathbb  R$. In this paper, we are interested in the situation when the inner function $U$ can be extended into a meromorphic function in $\mathbb C$. 
Such functions are called meromorphic inner functions (MIF) on the upper half-plane. They can be easily described via the standard Blaschke/singular factorization. All MIFs have the following form:
$$
U(z)=Ce^{iaz}\prod_{n=0}^\infty e^{i\alpha_n}\frac{z-w_n}{z-\bar w_n},\qquad (z\in\mathbb C_+),
$$
where $a$ is a non-negative constant, $w_n$ is a sequence of points in $\mathbb C_+$ tending to infinity as $n\to\infty$ and satisfying the Blaschke condition
$$\sum_{n=0}^\infty\frac{\Im{\rm{m}}(w_n)}{1+|w_n|^2}<\infty,$$
$C$ is a unimodular constant and $\alpha_n$ is a real number choosen so that 
$$
e^{i\alpha_n}=\frac{\left|\frac{i-w_n}{i-\overline{w_n}}\right|}{\frac{i-w_n}{i-\overline{w_n}}}.
$$
Associated to an inner function $U$ on $\mathbb C_+$, the model space $K_U$ is  defined by 
$$
K_U:= \mathcal{H}^2 \cap (U \mathcal{H}^2)^{\perp}.
$$
We also have the following equivalent definition
\begin{equation}\label{eq:defn-model-space}
K_{U} = \mathcal{H}^{2} \cap U \overline{\mathcal{H}^2},
\end{equation}
where $\overline{\mathcal{H}^2}$ is often regarded as the Hardy space of the lower half-plane. 

\subsection{Toeplitz operators and a characterization of multipliers}
Recall that to every $\varphi\in L^\infty(\mathbb  R)$, there corresponds the Toeplitz operator $T_\varphi:\mathcal H^2\longrightarrow \mathcal H^2$ defined by
$$
T_\varphi(f)=P_+(\varphi f),\qquad f\in\mathcal H^2,
$$
where $P_+$ is the orthogonal projection of $L^2(\mathbb R)$ onto $\mathcal H^2$. Using \eqref{eq:defn-model-space}, it is immediate to see that, when the function $U$ is inner, then 
\begin{equation}\label{eq:kernel-Toeplitz-antianalytic}
\ker\, T_{\overline U}=K_U.
\end{equation}
In \cite{fricain2016multipliers}, the following characterization of multipliers is proved. 
\begin{theorem}[Fricain--Hartmann-Ross]\label{thm:FHR}
Let $U$ and $V$ be inner functions with $|U'(x)| \asymp 1, x \in \mathbb R$, and let $\Phi$ be a function holomorphic on $\mathbb C_+$. Then the following are equivalent:
\begin{enumerate}
\item $\Phi\in \MM(U,V)$ 
\item $\Phi k_i\in\ker\, T_{\overline{b_{i}V} U}$ and $\sup_{x\in\R}\int_x^{x+1}|\Phi(t)|^2dt <\infty$.
\end{enumerate}
\end{theorem}
Note that the second condition appearing in (2) says that the measure $|\Phi(t)|^2\,dt$ is a Carleson measure for $K_U$ (see \cite[Theorem 5.1]{MR1784683}), ensuring that $\Phi K_U\subset\mathcal H^2$. 

As one see from Theorem~\ref{thm:FHR}, the non injectivity of a certain Toeplitz operator is necessary for the set of multipliers being non trivial. The problem of injectivity of Toeplitz operators is a classical problem in analysis, being related to completeness of exponential systems on $L^2(0,2\pi)$. In \cite{makarov2005meromorphic,makarov2010beurling}, Makarov--Poltoratski extended the theory of Beurling Malliavin density to model spaces related to MIF. See next section for a brief discussion on their results. We just mention here an easy result which shall be used below.  

\begin{lemma}\label{Lem1}
Let $B$ be a finite Blaschke product, $\Theta$ an inner function which is not a finite Blaschke product and let $1\leq p\leq \infty$. Then
$$
\ker\, T_{B\overline{\Theta}}\cap \mathcal H^p\neq \{0\}.
$$
\end{lemma}

\begin{proof}
Let us write 
$$
B(z)=\prod_{j=1}^k \left(\frac{z-w_j}{z-\overline{w_j}}\right)^{m_j}
$$
and define the linear map
$$T:\left |\begin{array}{cccl}

 & K_\Theta \cap \mathcal H^p & \longrightarrow & \mathbb C^{N} \cr
& f &\longmapsto & (f^{(s)}(\lambda_j))_{\substack{1\leq j\leq k \\ 1\leq s\leq m_j}}
\end{array} \right.$$
where $N=\sum_{1\leq j\leq k}m_j$.  Since $\Theta$ is not a finite Blaschke product, we know that $K_\Theta\cap\mathcal H^p$ is of infinite dimension and then $T$ is not one-to-one. Hence there exists a function $f\in K_\Theta\cap\mathcal H^p$, $f\not\equiv 0$,  such that for every $1\leq j\leq k$, $1\leq s\leq m_j$, $f^{(s)}(\lambda_j)=0$. We can write $f=Bg$ for some $g\in\mathcal H^p$. It remains to note that using \eqref{eq:kernel-Toeplitz-antianalytic}, we have
$$
T_{B\overline\Theta}(g)=P_+(\overline\Theta Bg)=T_{\overline\Theta}(f)=0.
$$
\end{proof}
%
%
%
\subsection{Beurling Malliavin densities}
Let $\Lambda \subset \mathbb C_+ \cup \R$. In \cite{makarov2005meromorphic,makarov2010beurling}, Makarov and Poltoratski connected the Beurling-Malliavin density of $\Lambda$ to the injectivity of the kernel of a related Toeplitz operator. We briefly recall some of these facts here. 
First, let $\Lambda \subset \R$ be a discrete sequence. We say that $\Lambda$ is \textit{strongly $a$-regular} if \begin{equation} \label{strongregularity}
\int_{\mathbb R}\frac{|n_{\Lambda}(x) - ax|}{1+x^2}dx<\infty,
\end{equation}
where $n_\Lambda$ is the counting function of $\Lambda$ defined by
$$
n_\Lambda(x)=\begin{cases}
\mbox{card }(\Lambda\cap [0,x]) & \mbox{if }x\geq 0 \\
-\mbox{card }(\Lambda\cap [x,0])& \mbox{if }x<0.
\end{cases}
$$

It is known (see \cite{poltoratski2015toeplitz,mitkovski2010polya}) that the interior Beurling-Malliavin (BM) density of a discrete sequence $\Lambda$ can be defined as
 $$ D_*(\Lambda) := \sup\{a: \exists \mbox{ strongly $a$-regular subsequence } \Lambda' \subset \Lambda\}. $$ 
 Similarly, the exterior BM density is defined as 
 $$ D^*(\Lambda) := \inf\{a: \exists \mbox{ strongly $a$-regular supsequence } \Lambda' \supset\Lambda\}. $$ 
 These definitions extend to the upper half-plane as well \cite{makarov2010beurling} in the following way. Let $\Lambda \subset \mathbb C_+$ be a discrete sequence, then $$D_*(\Lambda) := D_*(\Lambda^*),  $$
where $\Lambda^*:= \{\lambda^*: \lambda \in \Lambda, \Re \lambda \neq 0 \}$, $\lambda^* := [\Re (\lambda^{-1})]^{-1}$.

\begin{example} \label{infbp}
Let $\Lambda=\{n+i\}_{n\in \mathbb Z}$. Then $D_*(\Lambda)=D^*(\Lambda)= 1$.
\end{example} 
\noindent \emph{Proof.} For $n\in\mathbb Z^*$, we have $\lambda^*_n=[\Re(1/(n+i))]^{-1} = (n^2+1)/n$. The counting function of this sequence is odd and $n_{\Lambda^*}(x) = n$, for $x\in(n+1/n, n+1 + 1/(n+1))$, $n>0$. Then 
\begin{equation} \int_{2}^\infty \frac{|n_{\Lambda^*}(x) - x|}{1+x^2}dx= \sum_{n\geq 1} \int_{n+1/n}^{n+1+1/(n+1)} \frac{x-n}{1+x^2}dx \leq \sum_n \frac{3}{2} \left(1+\frac{1}{n+1}-\frac{1}{n}\right).\frac{1}{n^2+1} < \infty.\end{equation}
Thus $\Lambda^*$ is itself a $1-$ strongly regular sequence and so $D_*(\Lambda)=D^*(\Lambda)=1$.
\qed


It turns out that when $\Lambda$ is a discrete sequence on $\mathbb R$, then we can construct a MIF $\Theta$ with $\sigma(\Theta):=\{x\in \R: \Theta(x)=1\}=\Lambda$. Then it is proved in \cite{mitkovski2010polya,makarov2005meromorphic} that
$$D_*(\Lambda) = \frac{1}{2\pi}\inf\{a: \ker\, T_{S^a \overline \Theta}=\{0\} \}, $$
and $$D^*(\Lambda) = \frac{1}{2\pi}\sup\{a: \ker\, T_{\overline S^a \Theta }=\{0\} \}, $$
where $S$ is the singular inner function defined by $S(z)=e^{iz}$. 
 In terms of Toeplitz kernels, when $\Lambda$ is a Blaschke sequence in $\mathbb C_+$, 
 we can replace $\Theta$ by the Blaschke product $B_\Lambda$ 
 with zeroes on $\Lambda$, and we have
\begin{eqnarray} \label{bmdefinition}D_*(\Lambda) = \frac{1}{2\pi}\inf\{a: \ker\, T_{S^a \overline B_\Lambda}=\{0\} \},\\ 
D^*(\Lambda) = \frac{1}{2\pi}\sup\{a: \ker\, T_{\overline  S^a B_\Lambda}=\{0\} \}.\end{eqnarray}
Note that if $a>b$, then
\begin{equation}\label{eq:croissance-kernel}
\ker\,T_{\overline{S^b}B_\Lambda}\neq\{0\}\implies \ker\,T_{\overline{S^a}B_\Lambda}\neq\{0\}.
\end{equation}

\section{Main theorem and Examples}
In this section, we give a class of MIFs $U$ and $V$ for which the triviality of $\mathcal M(U,V)$ can be reduced to the injectivity of the Toeplitz operator $T_{U\overline V}$.  We end the section by showing examples of MIFs that fall into this category.

\begin{theorem}\label{mainthm}
Let $U$ and $V$ be MIFs with $|U'|\asymp 1$ on $\R$ and let $m:=\arg(U)-\arg(Vb_i)$ on $\R$. Suppose that either $m \not\in \widetilde L^1_\Pi$ or if $m=\tilde h$ for some  $h \in L^1_\Pi$, then  $e^{-h} \not\in L^1(\mathbb R)$. Then the following three conditions are equivalent.
\begin{enumerate}
\item $\dim$ $\ker\, T_{U\overline {Vb_i}}\geq 2$;
\item $\ker\, T_{U\overline V} \neq \{0\}$;
\item $\mathcal M(U,V) \neq \{0\}.$
\end{enumerate}
\end{theorem}

\begin{proof}
$(1)\implies (2)$: Since  $\dim$ $\ker\, T_{U\overline {Vb_i}}\geq 2$, we can find a function $\Psi \in \ker\, T_{U\overline {Vb_i}}$, $\Psi\not\equiv 0$, such that $\Psi(i)=0$. Then we can write $\Psi=b_i\Psi_1$ with $\Psi_1\in\mathcal H^2$. Since 
$$
0=T_{U\overline{Vb_i}}(\Psi)=T_{U\overline V}(\Psi_1),
$$
we have 
$\Psi_1 \in \ker\, T_{U\overline V}$ and $\Psi_1\not\equiv 0$.

$(2)\implies (3)$: Let $\Phi \in \ker\, T_{U\overline V} $ be non zero.  Then there is a function $g\in \mathcal H^2$ such that on $\R$, we have
$$
\Phi.U\overline V = \overline g. 
$$ 
Since
$$ 
\frac{\Phi}{z+i}. U\overline V. \overline {b_i}= \frac{\overline g}{z+i}.\frac{z+i}{z-i}=\overline{\left(\frac{g}{z+i}\right)}\in \overline{\mathcal H^2},
$$
then $\Phi k_i \in \ker\, T_{U\overline{Vb_i}}$. Moreover, using $\Phi\in\mathcal H^2$, we also have
$$
\sup_{x\in\mathbb R}\int_x^{x+1}|\Phi(t)|^2\,dt<\infty.
$$
Thus by Theorem~\ref{thm:FHR}, we deduce that $\Phi\in \mathcal M(U,V)$, which gives $(3)$.\\

$(3)\implies (1)$: Now assume that $\mathcal M(U,V)\neq \{0\}$. Then, according to Theorem~\ref{thm:FHR}, we know that  $\ker\, T_{U\overline {Vb_i}} \neq \{0\}$. We argue by contradiction and suppose that $\dim$ $\ker\, T_{U\overline {Vb_i}} = 1$. First let us prove that $\ker\, T_{U\overline {Vb_i}}$ is generated by an outer function. Indeed, let $f\in\mathcal H^2$ such that $\ker\, T_{U\overline {Vb_i}}=\mathbb C f$ and write $f=\Theta f_0$ where $\Theta$ and $f_0$ are respectively the inner and outer part of $f$. Notice that 
$$
T_{U\overline {Vb_i}}(f_0)=P_+(U\overline {Vb_i \Theta}f)=T_{\overline\Theta} T_{U\overline {Vb_i}}(f)=0,
$$
whence $f_0\in \ker \,T_{U\overline {Vb_i}}$ and there exists a $\lambda\in\mathbb C$ such that $f_0=\lambda f$. Thus $f$ is outer. 

By definition, there is a function $g \in \mathcal H^2$ such that on $\R$, $$U\overline {Vb_i} f = \overline g.$$ Let $g=g_ig_0$ be the inner-outer factorization of $g$. Then 
$$
U\overline {Vb_i} f g_i=\overline{g_0}.
$$ 
We deduce $g_i f \in \ker\, T_{U\overline {Vb_i}} $. Since $\ker\, T_{U\overline {Vb_i}}$ is generated by $f$, we necessarily get that $g_i$ is a constant of modulus one which we may of course assume to be one. Using that $f$ and $g_0$ are outer and satisfy $|f|=|g_0|$ on $\mathbb R$, we obtain that $g_0=f$, and thus 
\begin{equation}\label{eq:argument-1}
Uf=Vb_i \overline{f}.
\end{equation}
Since $f$ is an outer function that is square integrable on $\R$, there must exist a function $h_1 \in L^1_\Pi(\R)$ such that $f=e^{h_1+i\tilde h_1}$ on $\R$ and $|f|=e^{h_1} \in L^2(\R)$.  We compare the arguments in \eqref{eq:argument-1} which gives
$$
m=\arg(U)-\arg(Vb_i)=-2\tilde{h}_1=\tilde{h},$$
with $h=-2h_1$. But $h\in L^1_\Pi$ and $e^{-h} \in L^1(\R)$ a contradiction to our hypothesis. Thus  $\dim$ $\ker\, T_{U\overline {Vb_i}} \geq 2$.
\end{proof}

\begin{remark}\label{rem:hypothesis-main-result}
\rm{
For the assertions (1)$\implies$ (2) and (2) $\implies$ (3), we only use that $U$ and $V$ are MIFs with $|U'|\asymp 1$ on $\R$. It is only in the assertion (3)$\implies$(2) that we use the full hypothesis of the theorem.}
\end{remark}

It is natural to wonder for which MIFs $U$ and $V$ are the hypotheses of the above theorem  satisfied. We give examples here to illustrate that for many pairs of MIFs, this is indeed the case.

Let us denote the singular inner function $e^{iz}$ by $S(z)$. We know that MIFs have the form $S^aB_\Lambda$, where $a\geq 0$ and $B_\Lambda$ is a Blaschke product. So we assume that $U=S^aB_{\Lambda_1}$ and $V=S^bB_{\Lambda_2}$. 

\begin{example} Let $U=S^a$ and $V=S^b$. Then we have
$$
\mathcal M(U,V) \neq \{0\} \Longleftrightarrow b\geq a.
$$
\end{example} 
Indeed, if $b=a$ then $U=V$ and of course the constant functions are multipliers from $K_U$ into $K_V$. We may assume now that $a\neq b$.  Note that $m=\arg(U)-\arg(Vb_i)=(a-b)x+2\arctan(x)$ on $\R$. Since $2\arctan(x)\in L^\infty(\R)$ and $(a-b)x \not\in L^{o(1,\infty)}_{\Pi}$, the function $m$ does not belong to the space $\widetilde L^1_\Pi(\R)$. Of course, we also have $|U'|\asymp 1$ on $\R$. Therefore, we can apply Theorem~\ref{mainthm} which gives that $\mathcal M(U,V)\neq\{0\}$ if and only if $\ker\,T_{U\overline{V}}\neq \{0\}$. Since $T_{U\overline V}=T_{\overline{S^{b-a}}}$, we get from \eqref{eq:kernel-Toeplitz-antianalytic} that 
$$
b>a\implies \ker\,T_{U\overline V}=K_{S^{b-a}}\implies \mathcal M(U,V)\neq\{0\}.
$$
On the other hand, if $b<a$, then $T_{U\overline V}=T_{S^{a-b}}$ and the operator $T_{U\overline V}$ is thus one-to-one, which gives $\mathcal M(U,V)=\{0\}$. Note that the result can also be obtained from Crofoot's paper \cite{Crofoot}. See also \cite[Proposition 2.2]{fricain2016multipliers}.

\begin{example} Let $U=S^aB_{\Lambda_1}$ and $V=S^bB_{\Lambda_2}$ such that $a\neq b$ and $B_{\Lambda_1}$ and $B_{\Lambda_2}$ are finite Blaschke products. Then
$$
\mathcal M(U,V) \neq \{0\} \Longleftrightarrow b>a.
$$
\end{example} 
Indeed, note that $m=\arg(U)-\arg(Vb_i)=(a-b)x+\arg(B_{\Lambda_1})-\arg(B_{\Lambda_2})+2\arctan(x)$. Since $B_{\Lambda_1}$ and $B_{\Lambda_2}$ are finite Blaschke products, $\arg(B_{\Lambda_1})- \arg(B_{\Lambda_2}) +2 \arctan(x) \in L^\infty(\R)$. The function $(a-b)x \not\in L^{o(1,\infty)}_\Pi$. Thus, the function $m \not\in \widetilde L^1_\Pi(\R)$. We also have $|U'|\asymp 1$ on $\R$. Therefore, we can apply Theorem~\ref{mainthm} which gives that $\mathcal M(U,V)\neq\{0\}$ if and only if $\ker\,T_{U\overline{V}}\neq \{0\}$. Now if $b>a$, then $T_{U\overline V}=T_{\overline{\Theta}B_{\Lambda_1}}$ where $\Theta$ is the inner function defined by $\Theta=S^{b-a}B_{\Lambda_2}$. Hence, by Lemma~\ref{Lem1},  $\ker\,T_{U\overline V}\neq \{0\}$ and thus $\mathcal M(U,V)\neq\{0\}$. Note that Coburn's Lemma (see \cite[Page 318]{nikolski1986treatise}) implies that if $b>a$, then $\ker\,T_{V\overline U}=\{0\}$. By symmetry, we thus get that if $b<a$, then $\mathcal M(U,V)=\{0\}$.

\begin{example} \label{eg4}
Let $U=B_{\Lambda_1}S^a$ and $V=S^b$ where $a\geq 0$, $b>0$, $B_{\Lambda_1}$ is an infinite Blaschke product, and let 
$D:= D^*(\Lambda_1)$. Assume that  $|U'|\asymp 1$ on $\R$ and $b-a\neq 2\pi D$. Then  
$$
\mathcal M(U,V)\neq \{0\}  \Longleftrightarrow b-a>2\pi D.
$$
\end{example}
Indeed, if $b-a>2\pi D,$ then by definition of $D$, $\ker\, T_{U\overline V}=\ker\, T_{B_{\Lambda_1}\overline S^{b-a}}\neq \{0\}$. By Theorem~\ref{mainthm} and Remark~\ref{rem:hypothesis-main-result}, we deduce that $\mathcal M(U,V) \neq \{0\}$. 

Let us now assume that $b-a<2\pi D$. Using once more the definition of $D$, there exists $\beta>b-a$ such that $\ker\,T_{\overline{S^\beta}B_{\Lambda_1}}=\{0\}$. Since 
$$
T_{U\overline V}(f)=T_{S^{a-b}B_{\Lambda_1}}(f)=T_{\overline{S^\beta}S^{\beta+a-b}B_{\Lambda_1}}(f)=T_{\overline{S^\beta}B_{\Lambda_1}}(fS^{\beta+a-b}),\qquad f\in\mathcal H^2,
$$
we get that $\ker\,T_{U\overline V}=\{0\}$. It thus remain to prove that  that $U$ and $V$ satisfy the hypothesis of Theorem~\ref{mainthm} to get that $\mathcal M(U,V)=\{0\}$. So let $m=\arg(U)-\arg(Vb_i)=\arg(B_{\Lambda_1}S^a)-\arg(S^b b_i)$. We argue by contradiction and assume that $m=\tilde h$ for some $h \in L^1_\Pi(\R)$ and $e^{-h}\in L^1(\R)$. Let us choose an $\varepsilon >0$ such that $b-a+\varepsilon <2\pi D$. By Lemma~\ref{Lem1}, we know that $\ker\,T_{b_i \overline {S^\epsilon}}\cap\mathcal H^\infty\neq \{0\}$.  
Therefore, we use \cite[Proposition 3.14]{makarov2005meromorphic} to see that $\arg(b_i \overline {S^\epsilon})$ is of the form $-\alpha + \tilde h_2$, where $\alpha$ is the argument of a MIF, $h_2\in L^1_\Pi(\R)$ and $e^{-h_2} \in L^\infty(\R)$. Thus, \begin{eqnarray*}\arg(B_{\Lambda_1}\overline{S^{b-a+\epsilon}})&=& \arg(B_{\Lambda_1}\overline{S^{b-a}b_i}b_i\overline{S^\epsilon})\\ 
&=& \arg(B_{\Lambda_1}\overline{S^{b-a}b_i}) + \arg(b_i\overline{S^\epsilon})\\
&=& -\alpha + \widetilde{h+h_2},
\end{eqnarray*}
where $h+h_2 \in L^1_\Pi(\R)$ and $e^{-(h+h_2)} \in L^1(\R)$. Using \cite[Proposition 3.14]{makarov2005meromorphic} once more, we have that $\ker\, T_{B_{\Lambda_1}\overline{S^{b-a+\epsilon}}} \neq \{0\}$, and we get a contradiction between \eqref{eq:croissance-kernel} and the fact that $b-a+\epsilon < 2 \pi D$.

\begin{example}\label{eg5}
Let $U=S^a$ and $V=B_{\Lambda_2}S^b$ with $a>0$, $b\geq 0$. Let $D:=D_*(\Lambda_2)$ and assume that $a-b\neq 2\pi D$. By similar computations as above, we can say that  
$$
\mathcal M(U,V)\neq \{0\}\Longleftrightarrow a-b<2 \pi D. 
$$

\end{example}

\begin{corollary}Let $\Lambda=\{n+i\}_{n\in\mathbb Z}$, $U=S^a$, $V=B_\Lambda$ and assume that $a\neq 2\pi$.
Then 
$$
\mathcal M(U,V)\neq\{0\}\Longleftrightarrow a<2\pi.
$$
\end{corollary}
\begin{proof} 
By Example \ref{infbp},  we know that $D_*(\Lambda)=1$. Thus the conclusion follows from Example \ref{eg5}.
\end{proof}

\bibliographystyle{plain}
\bibliography{bibliography.bib}

\begin{thebibliography}{10}

\bibitem{MR1784683}
A.~D. Baranov.
\newblock Differentiation in the {B}ranges spaces and embedding theorems.
\newblock {\em J. Math. Sci. (New York)}, 101(2):2881--2913, 2000.
\newblock Nonlinear equations and mathematical analysis.

\bibitem{Crofoot}
R.~Bruce Crofoot.
\newblock Multipliers between invariant subspaces of the backward shift.
\newblock {\em Pacific J. Math.}, 166(2):225--246, 1994.

\bibitem{fricain2016multipliers}
Emmanuel Fricain, Andreas Hartmann, and William~T Ross.
\newblock Multipliers between model spaces.
\newblock {\em Studia Mathematica, to appear}, 2017.

\bibitem{GMR}
Stephan~Ramon Garcia, Robert T.~W. Martin, and William~T. Ross.
\newblock Partial order on partial isometries.
\newblock {\em J. Operator Theory}, 75(2):409--442, 2016.

\bibitem{makarov2005meromorphic}
N.~Makarov and A.~Poltoratski.
\newblock Meromorphic inner functions, toeplitz kernels and the uncertainty
  principle.
\newblock In {\em Perspectives in Analysis}, pages 185--252. Springer, 2005.

\bibitem{makarov2010beurling}
N.~Makarov and A.~Poltoratski.
\newblock Beurling-malliavin theory for toeplitz kernels.
\newblock {\em Inventiones Mathematicae}, 180(3):443--480, 2010.

\bibitem{MR2500010}
Javad Mashreghi.
\newblock {\em Representation theorems in {H}ardy spaces}, volume~74 of {\em
  London Mathematical Society Student Texts}.
\newblock Cambridge University Press, Cambridge, 2009.

\bibitem{mitkovski2010polya}
M.~Mitkovski and A.~Poltoratski.
\newblock P{\'o}lya sequences, toeplitz kernels and gap theorems.
\newblock {\em Advances in Mathematics}, 224(3):1057--1070, 2010.

\bibitem{nikolski1986treatise}
N.~K. Nikolski.
\newblock Treatise on the shift operator. spectral function theory. with an
  appendix by sv khrushch{\"e}v and vv peller, 1986.

\bibitem{poltoratski2015toeplitz}
A.~Poltoratski.
\newblock {\em Toeplitz Approach to Problems of the Uncertainty Principle},
  volume 121 of \textit{CBMS}.
\newblock American Mathematical Society, 2015.

\bibitem{MR3616198}
Dan Timotin.
\newblock On a preorder relation for contractions.
\newblock {\em Acta Sci. Math. (Szeged)}, 82(3-4):629--640, 2016.

\end{thebibliography}
\end{document}